\documentclass{amsart}
\usepackage{graphicx} 
\usepackage[table]{xcolor}

\usepackage{hyperref}   
\usepackage[square,numbers]{natbib}     
\usepackage{amsthm}     
\usepackage{amsmath}    
\usepackage{amssymb}    
\usepackage{amsxtra}    
\usepackage{eucal}      
\usepackage[bbgreekl]{mathbbol}   
\usepackage{todonotes}
\usepackage{comment}
\usepackage{tikz}

\usetikzlibrary{matrix}
\usetikzlibrary{patterns, shapes}
\usetikzlibrary{arrows}
\usetikzlibrary{calc,3d}
\usetikzlibrary{decorations,decorations.pathmorphing, decorations.pathreplacing}
\usetikzlibrary{through}
\tikzset{every picture/.style={baseline=-.65ex}}
\tikzset{ext/.style={circle, draw,inner sep=1pt},int/.style={circle,draw,fill,inner sep=1pt},nil/.style={inner sep=1pt}}
\tikzset{exte/.style={circle, draw,inner sep=3pt},inte/.style={circle,draw,fill,inner sep=3pt}}
\tikzset{diagram/.style={matrix of math nodes, row sep=3em, column sep=2.5em, text height=1.5ex, text depth=0.25ex}}
\tikzset{diagram2/.style={matrix of math nodes, row sep=0.5em, column sep=0.5em, text height=1.5ex, text depth=0.25ex}}
\tikzset{every loop/.style={draw}}

\newcommand{\tadpole}{
\begin{tikzpicture}[baseline=-.65ex]
\node[int] (v) at (0,0) {};
\draw (v) edge[loop] (v);
\end{tikzpicture}
}

\usepackage{tikz-cd}

\theoremstyle{plain}

\newtheorem{conj}[subsection]{Conjecture}
\newtheorem{prop}[subsection]{Proposition}
\newtheorem{lemma}[subsection]{Lemma}
\newtheorem{cor}[subsection]{Corollary}

\theoremstyle{definition}
\newtheorem{rem}[subsection]{Remark}

\DeclareMathOperator{\vspan}{span}
\newcommand{\G}{\mathsf{G}}
\newcommand{\GC}{\mathsf{GC}}

\DeclareMathOperator{\gr}{gr}
\DeclareMathOperator{\rk}{rank}
\DeclareMathOperator{\vdim}{dim}
\newcommand{\lp}{\text{-loop}}

\newcommand{\grt}{\mathfrak{grt}}
\newcommand{\F}{\mathbb{F}}
\newcommand{\Q}{\mathbb{Q}}

\title{The 11-loop graph cohomology}
\author{Thomas Willwacher}
\address{Department of Mathematics, ETH Zurich, Rämistrasse 101, 8092 Zurich, Switzerland}

\thanks{This work has been partially supported by the NCCR Swissmap, funded by the Swiss National Science Foundation}
\date{}

\begin{document}

\maketitle

\begin{abstract}
    We compute the Kontsevich graph cohomology in loop order 11, and for some degrees in higher loop order. We apply our results to discuss a conjecture of F. Brown, providing a counterexample to a strong version of the conjecture.
\end{abstract}

\section{Introduction}
The Kontsevich (commutative) graph complexes $\G_n$ are differential graded vector spaces spanned by linear combinations of at least trivalent connected graphs. More precise, for $n$ a number, $\G_n$ is generated as a graded vector space by pairs $(\Gamma,o)$ with $\Gamma$ a connected graph with $\geq 3$-valent vertices and without self-edges (tadpoles) and $o$ an orientation datum, which is either an ordering of the set of edges if $n$ is even or an ordering of the set of vertices and half-edges if $n$ is odd. One imposes two relations among those generators:
\begin{enumerate}
\item (Isomorphism) For $\phi:\Gamma\to \Gamma'$ an isomorphism of graphs we set $(\Gamma,o) = (\Gamma',\phi_*o)$ with $\phi_*o$ the orientation datum on $\Gamma'$ naturally induced by $\phi$ from $o$.
\item (Orientation change) For $\sigma$ a permutation of the set of edges (resp. half-edges and vertices) and $\pi o$ the appropriately permuted orientation datum we set 
$(\Gamma,o)=\mathrm{sgn}(\sigma) (\Gamma, \sigma o)$.
\end{enumerate}
\[
    \begin{tikzpicture}[baseline=-.65ex, scale=.5]
\node[int] (c) at (0,0){};
\node[int] (v1) at (0:1) {};
\node[int] (v2) at (72:1) {};
\node[int] (v3) at (144:1) {};
\node[int] (v4) at (216:1) {};
\node[int] (v5) at (-72:1) {};
\draw (v1) edge (v2) edge (v5) (v3) edge (v2) edge (v4) (v4) edge (v5)
      (c) edge (v1) edge (v2) edge (v3) edge (v4) (c) edge (v5);
\end{tikzpicture}
+
3\,
    \begin{tikzpicture}[baseline=-.65ex, scale=.5]
\node[int] (c) at (0.7,0){};
\node[int] (v1) at (0,-1) {};
\node[int] (v2) at (0,1) {};
\node[int] (v3) at (2.1,-1) {};
\node[int] (v4) at (2.1,1) {};
\node[int] (d) at (1.4,0) {};
\draw (v1) edge (v2) edge (v3)  edge (d) edge (c) (v2) edge (v4) edge (c) (v4) edge (d) edge (v3) (v3) edge (d) (c) edge (d);
\end{tikzpicture} \, \in \G_n,
  \]
We give $G_n$ the structure of a graded vector space by declaring $(\Gamma,o)$ to be of cohomological degree 
$$
|(\Gamma,o)| = (n-1)\#\text{edges}- n(\#\text{vertices}-1).
$$
We define a differential on $\G_n$ as
\[
d\Gamma =\sum_e (\Gamma/e, o_e)
\]
where the sum is over edges, $\Gamma/e$ is obtained by contracting the edge $e$ and $o_e$ is a natural orientation datum associated to $o$ and $e$, see \cite[Section 2]{BrunWillwacher}. 
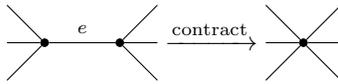
\[
\begin{tikzpicture}
  \node[int] (v) at (0,0) {};
  \node[int] (w) at (1,0) {};
  \draw (v) edge node[above] {$\scriptstyle e$} (w) edge +(-.5,.5) edge +(-.5,-.5) edge +(-.5,0)
  (w) edge +(.5,.5) edge +(.5,-.5) edge +(.5,0);
\end{tikzpicture}
\xrightarrow{\text{contract}}
\begin{tikzpicture}
  \node[int] (v) at (0,0) {};
  \draw (v) edge +(-.5,.5) edge +(-.5,-.5) edge +(-.5,0)
  edge +(.5,.5) edge +(.5,-.5) edge +(.5,0);
\end{tikzpicture}
\]
The differential does not alter the loop order of graphs, and hence we have a direct sum decomposition of dg vector spaces according to loop order
\[
\G_n = \bigoplus_{g\geq 3} \G_n^{g\lp}.
\]
Each complex $\G_n^{g\lp}$ is of finite, though rapidly growing dimension. The dependence on $n$ in the definition of the graph complex is somewhat mild in that 
\[
\G_n^{g+2\lp} \cong \G_n^{g\lp}[2g].
\]
Hence all $\G_n$ for even $n$ are isomorphic up to loop-order-dependent regrading, as well as all $\G_n$ for odd $n$. However, no substantial link is known to the author between the cohomologies of $\G_2$ and $\G_3$ for example, despite both having closely related Euler characteristics \cite{BorinskyZagier}. 

The graph complex $\G_n$ is in fact a dg Lie coalgebra, and hence it is natural to work with the dual dg Lie algebra 
\[
\GC_n := \G_n^*
\]
as well. Elements of $\GC_n$ can also be considered as series of graphs with an orientation datum, with the same equivalence relations as above, but with the dual differential given combinatorially by splitting vertices instead of contracting edges. 
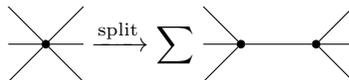
\[
\begin{tikzpicture}
  \node[int] (v) at (0,0) {};
  \draw (v) edge +(-.5,.5) edge +(-.5,-.5) edge +(-.5,0)
  edge +(.5,.5) edge +(.5,-.5) edge +(.5,0);
\end{tikzpicture}
\xrightarrow{\text{split}}
\sum \,
\begin{tikzpicture}
  \node[int] (v) at (0,0) {};
  \node[int] (w) at (1,0) {};
  \draw (v) edge (w) edge +(-.5,.5) edge +(-.5,-.5) edge +(-.5,0)
  (w) edge +(.5,.5) edge +(.5,-.5) edge +(.5,0);
\end{tikzpicture}
\]
Also, in our conventions dualization negates degrees, and a generator $(\Gamma,o)$ of the dual graph complex $\GC_n$ is hence considered of cohomological degree 
$$
|(\Gamma,o)| = -(n-1)\#\text{edges}+ n(\#\text{vertices}-1).
$$

Since its introduction by Kontsevich in the 1990s the graph homology $H(\GC_n)\cong H(\G_n)^*$ has seen numerous applications in algebra, topology and geometry \cite{CGP1,Willwachergrt,FTW}. However, there is still no satisfying understanding of its structure at large. Hence numerical computations are of considerable interest. By earlier works \cite{BNMK, BrunWillwacher} we know $H(\GC_n^{g\lp})$ for $g\leq 10$. The main contribution of this work is the computation of $H(\GC_n^{11\lp})$, and several entries beyond. Our current knowledge of $H(\GC_n^{g\lp})$, including the results of the present work, is summarized in Tables \ref{tab:GC2} and \ref{tab:GC3}.

\begin{table}[h]
  \centering
    \begin{tabular}{|c|c|c|c|c|c|c|c|c|c|c|c|c|c|c|c|c|}
    \hline
    \rowcolor{gray!50} 
     k,g & 3 & 4 & 5 & 6 & 7 & 8 & 9 & 10 & 11 & 12 & 13 & 14 & 15 & 16 & 17 & 18 \\ \hline
     11 & - & - & - & - & - & - & - & - & - & - & - & $\leq 1$ & ? & & &  \\ \hline
     10 & - & - & - & - & - & - & - & - & - & - & $\leq 1$ & ? &  & & &  \\ \hline
     9 & - & - & - & - & - & - & - & - & - & 0 & ? &  &  &  & &  \\ \hline
    8 & - & - & - & - & - & - & - & - & 0 & ? &  &  &  &  & &  \\ \hline
    7 & - & - & - & - & - & - & - & 1 & 0 & ? &  &  &  &  & &  \\ \hline
    6 & - & - & - & - & - & - & 0 & 0 & 1 & ? &  &  &  &  & &  \\ \hline
    5 & - & - & - & - & - & 0 & 0 & 0 & 0 & ? &  &  &  &  & &  \\ \hline
    4 & - & - & - & - & 0 & 0 & 0 & 0 & 0 & ? &  &  &  &  & &  \\ \hline
    3 & - & - & - & 1 & 0 & 1 & 1 & 2 & 2 & 2 & 4 & ? &  &   & & \\ \hline
    2 & - & - & 0 & 0 & 0 & 0 & 0 & 0 & 0 & 0 & 0 & ? &  &  & &  \\ \hline
    1 & - & 0 & 0 & 0 & 0 & 0 & 0 & 0 & 0 & 0 & 0 & ? &  &  & &  \\ \hline
    0 & 1 & 0 & 1 & 0 & 1 & 1 & 1 & 1 & 2 & 2 & 3 & 3 & 4 & 5 & 7 &$\dots$ \\ \hline
  \end{tabular}
  
  \medskip
  
  \begin{tabular}{|c|c|c|c|c|c|c|c|c|c|c|c|c|c|c|c|c|}
    \hline
    \rowcolor{gray!50} 
     k,g & $\dots$ & 18 & 19 & 20 & 21 & 22 & 23 & 24 & 25 & 26 & 27 & 28 & 29 & 30 \\ \hline
    0 & $\dots$ & 8 & 11 & 13 & 17 & 21 & 28 & 34 & 45 & 56 & 73 & 92 & 120 & ?
     \\ \hline
  \end{tabular}
  \caption{\label{tab:main res even n}The dimensions $dim H^k(\GC_2^{g-loop})$ of the graph cohomology. Note that $H^0(\GC_2^{g-loop}) =\gr_g\grt_1$ is the weight $g$ part of the Grothendieck-Teichmüller Lie algebra, whose dimension is known up to $g=29$, see \cite{NaefWillwacherKV, datamine}. The entries to the right of the question marks, including the question marks, are unknown. The values for $k<0$ are zero by known degree bounds. The same holds for $k>g-3$, since $\GC_2^{g-loop}$ is zero beyond that degree.}
  \label{tab:GC2}
\end{table}

\begin{table}[h]
  \centering
  \begin{tabular}{|c|c|c|c|c|c|c|c|c|c|c|c|c|c|c|c|c|}
    \hline
    \rowcolor{gray!50} 
     k,g & 2 & 3 & 4 & 5 & 6 & 7 & 8 & 9 & 10 & 11 & 12 & 13 & 14 & 15 &16 & 17 \\ \hline
    -3 & 1 & 1 & 1 & 2 & 2 & 3 & 4 & 5 & 6 & 8 & 9 & 11 & 13 & 
    ?
    & &\\ \hline
    -4 & - & - & 0 & 0 & 0 & 0 & 0 & 0 & 0 & 0 & ? & ? &  & & & \\ \hline
    -5 & - & - & - & 0 & 0 & 0 & 0 & 0 & 0 & 0 & ? & ? &  &  & & \\ \hline
    -6 & - & - & - & - & 1 & 1 & 2 & 3 & 5 & $7^\dagger$ & ? & ? &  &  & & \\ \hline
    -7 & - & - & - & - & - & 0 & 0 & 0 & 0 & $1^\dagger$ & ? &  &  &  & & \\ \hline
    -8 & - & - & - & - & - & - & 0 & 0 & 0 & 0 & ? &  &  &  & & \\ \hline
    -9 & - & - & - & - & - & - & - & 0 & 0 & 1 & ? &  &  &  & & \\ \hline
   -10 & - & - & - & - & - & - & - & - & 0 & 0 & ? &  &  &  & & \\ \hline
   -11 & - & - & - & - & - & - & - & - & - & 0 & ? &  &  &  & & \\ \hline
   -12 & - & - & - & - & - & - & - & - & - & - & $0^*$ & ? &  &  & & \\ \hline
   -13 & - & - & - & - & - & - & - & - & - & - & - & $0^*$ & ?  &  & &\\ \hline
   -14 & - & - & - & - & - & - & - & - & - & - & - &  - & $0^*$ & ?& &\\ \hline
   -15 & - & - & - & - & - & - & - & - & - & - & - &  - &  - & $0^*$ & ?&\\ \hline
   -16 & - & - & - & - & - & - & - & - & - & - & - &  - & - &- & $0^*$ & ?\\ \hline
   -16 & - & - & - & - & - & - & - & - & - & - & - &  - & - &- &- & $0^*$ \\ \hline
  \end{tabular}
  \caption{\label{tab:main res odd n}The dimensions $dim H^k(\GC_3^{g-loop})$ of the odd graph cohomology. 
  For the entries of the first line some lower bounds are known for all $g$, see \cite{Vogel}.
  The entries $(-)^\dagger$ for $g=11$, $k=-6,-7$ are not guaranteed to be exact $\Q$-dimensions, see the discussion in the text -- they are at least upper bounds, though, and if not tight the numbers $7,1$ would necessarily by $6,0$. 
  The entries $0^*$ are computed in the upcoming paper \cite{NaefWillwacherOKV}. They are not part of the present work, but nevertheless listed here for convenience of the reader.}
  \label{tab:GC3}
\end{table}

\section{Discussion of the entries of tables \ref{tab:GC2} and \ref{tab:GC3}}
\label{sec:cohom discussion}
The main tables \ref{tab:GC2} and \ref{tab:GC3} subsume data from multiple sources.
We shall discuss here the individual entries and where they come from, and give proofs if necessary. First, we note that by prior works \cite{BrunWillwacher, BNMK} the entries for $g\leq 10$ have been computed numerically.
The main contribution of the present work is the numerical computation of the entries for $g=11$.
To be precise, our numerical computations covers the cases of $g=11$, all $k$ for $n=3$, and $g=11$, $k\geq 5$ for $n=2$. The other entries in $g=11$ can be inferred theoretically as we shall discuss shortly.

First, we have the well-known vanishing statements that $H^k(\GC_2)=0$ for $k<0$ or $k>g-3$, and $H^k(\GC_3)=0$ for $k>-3$ or $k<-g$ (if $k\geq 3$) \cite{Willwachergrt}.
It is known by an earlier work of the author that $H^0(\GC_2)\cong \grt_1$ is identified with the Grothendieck-Teichmüller Lie algebra, which by work of Brown \cite{Brown} receives an injective map from a free Lie algebra in symbols $\sigma_3$, $\sigma_5$, $\dots$ (with $\sigma_{2k+1}$ of loop order $2k+1$). The Deligne-Drinfeld conjecture states that this injection is an isomorphism. This conjecture has been verified numerically up to loop order 29 \cite{NaefWillwacherKV}, hence the dimension of the Grothendieck-Teichmüller Lie algebra is known as far, and displayed in \ref{tab:GC2}. There is an odd analog of the Grothendieck-Teichmüller Lie algebra studied in the upcoming work \cite{NaefWillwacherOKV}, and as a consequence one can also compute the bottom cohomology $H^{-g}(\GC_3^{g\lp})$ for $g\leq 17$, the result being mostly zero.

Next consider the top cohomology $H^{-3}(\GC_3)$. This is a version of the space of Vassiliev invariants in knot theory, and has been studied extensively in the past \cite{Vogel, KneisslerNumber, BroadhurstV}. It had been known up to loop order 12, with lower bounds for all loop orders, and a conjectured dimension formula. Our contribution here is to extend the known range to $g\leq 14$, by computing upper bounds on the cohomology similarly to \cite{KneisslerNumber}, that happen to agree with the known lower bounds \cite[Table 1]{BroadhurstV}, \cite{Vogel}. The details on the top degree computation can be found in section \ref{sec:top degree} below.
The computation of the upper bounds for the top degree cohomology also applies mutatis mutandis to $\GC_2$, so that we similarly obtain the top degree entries in $g=12,13,14$ in Table \ref{tab:GC2}.
Unfortunately, in the cases $g=13,14$ no matching lower bound is known, though we conjecture that our upper bound is tight.

Finally turn to the remaining entries in Table \ref{tab:GC2} in $g=11,12,13$ and $k=1,2,3,4$.
In order to obtain these entries we use the spectral sequence introduced in \cite{KWZ} and further studied in \cite{WillwacherTriconnected}, see also Section \ref{sec:reps} below.
Precisely, there is a spectral sequence 
\[
E^1=H(\GC_0) \Rightarrow \bigoplus_k \Q \sigma_{2k+1}
\]
that starts from the cohomology $H^\bullet(\GC_0)$ (with $H^k(\GC_0^{g\lp})\cong H^{k-2g}(\GC_2^{g\lp})$). The upshot is that due to this spectral sequence, classes in $H(\GC_2)$ must cancel in pairs, with the exception of one class in degree 0 and each loop order $3,5,7,\dots$. A class in loop order $g$ and degree $k$ can cancel classes in loop order and degree $(g',k')=(g+i,k-2i+1)$ for $i=1,2,\dots$. Furthermore, as shown in \cite[Corollary 4.4]{WillwacherTriconnected}, no cancellation is possible for $i=1,3,5$, since the differential vanishes on the respective pages of the spectral sequence.

Now consider the entry $(g,k)=(11,1)$ of Table \ref{tab:GC2}. Potential nonzero classes in this loop order and degree combination cannot cancel other classes in the spectral sequence since the cohomology vanishes in negative degrees. Hence if the entry was nonzero, any nontrivial class would need to be killed in the spectral sequence, and then necessarily by some other class in $(g',k')=(11-i, 2i)$ for $i=2,4,6,7,\dots$. But there is no cohomology in those positions so that we need to have $H^1(\GC_2^{11\lp})=0$.
The same argument works to show that 
\[
H^2(\GC_2^{11\lp})=H^1(\GC_2^{12\lp})= H^2(\GC_2^{12\lp})=H^1(\GC_2^{13\lp})=H^2(\GC_2^{13\lp})=0.
\]
Next consider the entry $(g,k)=(11,3)$. Again no class in this spot can be killed in the spectral sequence. since there is no cohomology in the single possible position $(9,6)$.
Hence all classes in $(g,k)=(11,3)$ must kill other classes, and these other classes must necessarily be in the location $(12,0)$. But $H^0(\GC_2^{12\lp})=\vspan\{[\sigma_3,\sigma_9], [\sigma_5,\sigma_7]\}$ is 2-dimensional. Furthermore, the two generating classes can only be canceled in the spectral sequence from the $(11,3)$-position.
Hence we have shown that $H^3(\GC_2^{11\lp})\cong \Q^2$ is 2-dimensional.
Finally, we know the Euler characteristic \cite{WillwacherZivkovic, BorinskyZagier}, which is $1$ for $\GC_2^{11\lp}$. Hence it suffices to compute numerically the cohomology $H^k(\GC_2^{11\lp})$ for $k=5,6,7,8$ (see Table \ref{tab:comp res} for the numerical results), and we can then infer the case $k=4$, which computes to 0 as shown in Table \ref{tab:GC2}.

Next consider the entry $(g,k)=(13,3)$. One class (i.e., a one dimensional subspace) can be killed in the spectral sequence by the class at position $(11,6)$.
The class at position $(11,6)$ in fact has no other classes it could kill (the next allowed position $(15, -1)$ is already in the vanishing range), and there are no classes that could kill it in the spectral sequence. Hence we deduce that the $(11,6)$-class kills one class in the $(13,3)$-position on the $E^2$-page of the spectral sequence.
Furthermore, classes in the $(13,3)$ position can kill (only) classes in the $(15,0)$-position. But $H^0(\GC_2^{15\lp})=\vspan\{\sigma_{15},[\sigma_3,[\sigma_3,\sigma_9]],[\sigma_3,[\sigma_5,\sigma_7]], [\sigma_5,[\sigma_3,\sigma_7]]  \}$ is 4-dimensional with one class $\sigma_{15}$ surviving to the $\infty$-page of the spectral sequence.
We hence retain 3 classes to be killed. But the only allowed position they can be killed from is $(13,3)$, as $H^7(\GC_2^{11\lp})=0$. Alternatively, this also follows without knowing $H^7(\GC_2^{11\lp})$ from the discussion in \cite[Section 4]{WillwacherTriconnected} and the fact that all classes to be killed lie in the ideal of $\grt_1$ generated by $\sigma_3$.
In any case we conclude that $H^3(\GC_2^{13\lp})\cong \Q^4$.

Finally consider the entry $(g,k)=(12,3)$. No class here can be killed in the spectral sequence by similar degree reasons as above, so all classes need to kill other classes, and those other classes need to live in $(g',k')=(14,0)$. We know that $H^0(\GC_2^{14\lp})\cong \vspan\{[\sigma_3,\sigma_{11}], [\sigma_5,\sigma_9], [\sigma_3,[\sigma_3,[\sigma_3,\sigma_5]]]\}\cong \Q^3$. All three classes need to be killed in the spectral sequence, by classes in position $(12,3)$ or $(10,7)$.
At this point we need one new input: By Corollary \ref{cor:X10 survives} to be shown below we know that the class at $(10,7)$ kills one class at $(14,0)$. This is not clear from the table alone, since it could potentially also kill at the yet unknown entry $(12,4)$.
However, with the additional information we readily conclude that $H^3(\GC_2^{12\lp})\cong \Q^2$.

\section{The top degree cohomology}
\label{sec:top degree}
We next consider the top degree cohomology 
\[
H^{top}(\GC_n^{g\lp}) := H^{(3-n)g -3}(\GC_n^{g\lp}).
\]
Concretely, the top degree elements of $\GC_n^{g\lp}$ are linear combinations of trivalent graphs of loop order $g$ with $3g-3$ edges and $2g-g$ vertices.
They are automatically cocycles. The coboundaries are obtained by applying the differential to graphs with exactly one 4-valent vertex and all other vertices trivalent.
This produces the IHX relations
\begin{equation}\label{equ:IHX}
\begin{tikzpicture}
    \node[int] (v) at (0,.3) {};
    \node[int] (w) at (0,-.3) {};
    \draw (v) edge (w) edge +(-.5,.2) edge +(.5,.2) 
    (w) edge +(-.5,-.2) edge +(.5,-.2) ;
\end{tikzpicture}
\ + \
\begin{tikzpicture}
    \node[int] (v) at (.3,0) {};
    \node[int] (w) at (-.3,0) {};
    \draw (v) edge (w) edge +(0.2,-.5) edge +(0.2,.5) 
    (w) edge +(-0.2,-.5) edge +(-0.2,.5) ;
\end{tikzpicture}
\ + \
\begin{tikzpicture}
    \node[int] (v) at (.2,-.1) {};
    \node[int] (w) at (-.2,-.1) {};
    \draw (v) edge (w) edge (.5,-.5) edge (-.5,.5) 
    (w) edge (-.5,-.5) edge (.5,.5) ;
\end{tikzpicture}
=0.
\end{equation}
For odd $n$, specifically $n=3$, the top degree cohomology $H^{-3}(\GC_3)$ is fairly well studied. Many nontrivial classes are produced by Chern-Simons theory, and we hence obtain lower bounds on the cohomology, see \cite{Vogel}. 
Furthermore, for not too large $g$, upper bounds have been produced by Kneissler \cite{KneisslerNumber}.
To recall Kneissler's results, let us first introduce, for any permutation $\pi\in S_N$, the barrel graph (drawn for $N=5$)
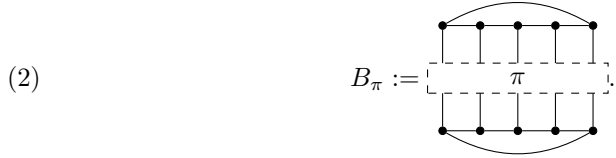
\begin{equation}\label{equ:Bpi def}
B_\pi :=
\begin{tikzpicture}
    \node[int] (v1) at (0,.7) {};
    \node[int] (v2) at (.5,.7) {};
    \node[int] (v3) at (1,.7) {};
    \node[int] (v4) at (1.5,.7) {};
    \node[int] (v5) at (2,.7) {};
    \node[int] (w1) at (0,-.7) {};
    \node[int] (w2) at (.5,-.7) {};
    \node[int] (w3) at (1,-.7) {};
    \node[int] (w4) at (1.5,-.7) {};
    \node[int] (w5) at (2,-.7) {};
    \draw (v1) edge[bend left] (v5) (w1) edge[bend right] (w5)
    (v2) edge (v1) edge (v2) (v3) edge (v2) edge (v4) (v4) edge (v5)
    (w2) edge (w1) edge (v2) edge (w2) (w3) edge (v3) edge (w2) edge (w4) (w4) edge (v4) edge (w5) (w1) edge (v1) (w5) edge (v5);
    \draw[dashed, fill=white] (-.2,-.2) rectangle (2.2,.2);
    \node at (1,0) {$\pi$};
\end{tikzpicture}.
\end{equation}
The graph $B_\pi$ has $2N$ vertices, $N$ on the upper rim and $N$ on the lower rim, and the upper $N$ vertices are connected to the lower $N$ vertices according to the permutation $\pi$. There are $3N$ edges and the loop order is hence $g=N+1$.

\begin{prop}[Kneissler \cite{KneisslerNumber}]
\label{prop:KneisslerOrig}
Any top degree element of $\GC_3^{g\lp}$ is cohomologous to a linear combination of barrel diagrams $B_\pi$, for $\pi$ ranging over elements of $S_{g-1}$.
\end{prop}
Next, Kneissler derives a set of relations on the barrel diagrams that then allow to compute an upper bound on the dimension of $H^{top}(\GC_3^{g\lp})$.
Here we shall proceed similarly. We treat the cases of even and odd $n$ uniformly and extend Kneissler's result to the following:
\begin{prop}
For any $n$ and for $g\geq 4$, any top degree element of $\GC_n^{g\lp}$ is cohomologous to a linear combination of barrel diagrams $B_\pi$, for $\pi$ ranging over elements of $S_{g-1}$.
\end{prop}
\begin{proof}
    The case of odd $n$ is Proposition \ref{prop:KneisslerOrig} shown by Kneissler, hence we assume that $n$ is even. We then use a variation of Kneissler's argument \cite[section 5.1]{KneisslerNumber}, with a modified last step.
    Start with an arbitrary trivalent graph $\Gamma$. It has at least one cycle. Hence we can single out an arbitrary cycle $S$ in the graph, so that the graph looks like this, with the chosen cycle drawn at the bottom with fat edges:
    
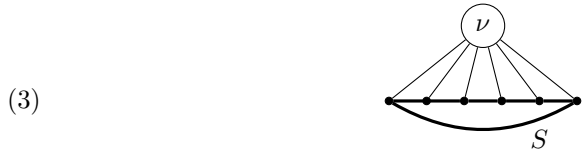
\begin{equation}\label{equ:std form}
    \begin{tikzpicture}
        \node[int] (u) at (0,0) {}; 
        \node[int] (u1) at (.5,0) {}; 
        \node[int] (u2) at (1,0) {}; 
        \node[int] (u3) at (1.5,0) {}; 
        \node[int] (u4) at (2,0) {}; 
        \node[int] (u5) at (2.5,0) {};
        \node at (2,-.5) {$S$};
        \node[draw, circle] (nu) at (1.25, 1) {$\nu$};
        \draw[very thick] (u) edge[bend right] (u5) edge (u2) (u3) edge (u2) edge (u4) (u4) edge (u5);
        \draw (nu) edge (u) edge (u1) edge (u2) edge (u3) edge (u4) edge (u5);
    \end{tikzpicture}
    \end{equation}
    The complement $\nu$ of our chosen cycle we call the remainder.

    Case 0): Suppose that the remainder $\nu$ is connected with a single cycle containing all vertices of $\nu$. Then $\Gamma$ is already a barrel graph $B_\pi$, for a suitable permutation $\pi$.

    Case 1): Suppose that $\nu$ is connected and of loop order 1. Then we may use the IHX relations \eqref{equ:IHX} to move vertices either to our base cycle $S$ or to the unique cycle of $\nu$, until we end up with a linear combination of graphs as in Case 0), i.e., barrel graphs.

    Case 2): Suppose that $\Gamma$ has the special shape 
    \[
    C_\pi = 
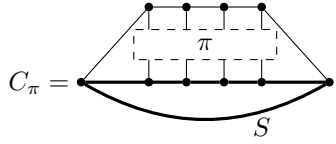
    \begin{tikzpicture}
        \node[int] (u) at (-.4,0) {}; 
        \node[int] (u1) at (.5,0) {}; 
        \node[int] (u2) at (1,0) {}; 
        \node[int] (u3) at (1.5,0) {}; 
        \node[int] (u4) at (2,0) {}; 
        \node[int] (v1) at (.5,1) {}; 
        \node[int] (v2) at (1,1) {}; 
        \node[int] (v3) at (1.5,1) {}; 
        \node[int] (v4) at (2,1) {}; 
        \node[int] (u5) at (2.9,0) {};
        \node at (2,-.6) {$S$};
        \draw[very thick] (u) edge[bend right] (u5) edge (u2) (u3) edge (u2) edge (u4) (u4) edge (u5);
        \draw (v1) edge (u) edge (v2) edge (u1) 
        (v2) edge (v3) edge (u2) 
        (v3) edge (u3) edge (v4)
        (v4) edge (u4) edge (u5);
        \draw[dashed, fill=white] (.3, .3) rectangle (2.2,.7);
        \node at (1.25,.5) {$\pi$};
    \end{tikzpicture}
    \]
    for $\pi$ some permutation. (Kneissler calls such graphs circle diagrams.)
    Then using the IHX relation to permute the attachment points of the legs of $\nu$ to $S$, we see that for any $\pi,\pi'\in S_n$ we have that $C_\pi\equiv\pm C_{\pi'}$ modulo IHX relations and graphs as in Case 1).
    But if the loop order $g$ is even, then $C_{id}=0$ by symmetry. If $g\geq 5$ is odd then $C_\tau=0$ by symmetry for any pair transposition $\tau$. Note that the case $g=3$ is excluded in the Proposition, so we find that all graphs of the form $C_\pi$ can be written as linear combinations of barrel graphs modulo the IHX relation.
    
    Case 3): Suppose that $\Gamma$ has the more general form 
    \[
    C_{\pi_1,\pi_2,\dots,\pi_k} = 
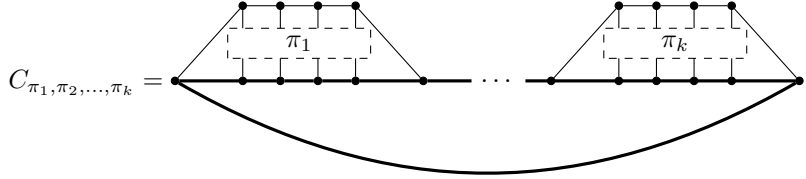
    \begin{tikzpicture}
        \node[int] (u) at (-.4,0) {}; 
        \node[int] (u1) at (.5,0) {}; 
        \node[int] (u2) at (1,0) {}; 
        \node[int] (u3) at (1.5,0) {}; 
        \node[int] (u4) at (2,0) {}; 
        \node[int] (v1) at (.5,1) {}; 
        \node[int] (v2) at (1,1) {}; 
        \node[int] (v3) at (1.5,1) {}; 
        \node[int] (v4) at (2,1) {}; 
        \node[int] (u5) at (2.9,0) {};
        \node (dd) at (3.9,0) {$\dots$};
        \begin{scope}[xshift=5cm]
        \node[int] (bu) at (-.4,0) {}; 
        \node[int] (bu1) at (.5,0) {}; 
        \node[int] (bu2) at (1,0) {}; 
        \node[int] (bu3) at (1.5,0) {}; 
        \node[int] (bu4) at (2,0) {}; 
        \node[int] (bv1) at (.5,1) {}; 
        \node[int] (bv2) at (1,1) {}; 
        \node[int] (bv3) at (1.5,1) {}; 
        \node[int] (bv4) at (2,1) {}; 
        \node[int] (bu5) at (2.9,0) {};
        \draw (bv1) edge (bu) edge (bv2) edge (bu1) 
        (bv2) edge (bv3) edge (bu2) 
        (bv3) edge (bu3) edge (bv4)
        (bv4) edge (bu4) edge (bu5);
        \draw[dashed, fill=white] (.3, .3) rectangle (2.2,.7);
        \node at (1.25,.5) {$\pi_k$};
        \end{scope}
        \draw[very thick] (u) edge[bend right] (bu5) edge (u2) (u3) edge (u2) edge (u4) (u4) edge (u5) (dd) edge (u5) edge (bu5);
        \draw (v1) edge (u) edge (v2) edge (u1) 
        (v2) edge (v3) edge (u2) 
        (v3) edge (u3) edge (v4)
        (v4) edge (u4) edge (u5);
        \draw[dashed, fill=white] (.3, .3) rectangle (2.2,.7);
        \node at (1.25,.5) {$\pi_1$};
    \end{tikzpicture}
    \]
    Then using the IHX relation at the junction between the first two components
    % ...
    we can reduce $k$, and by induction we have reduced to Case 2).

    Case 4): 
    Suppose our graph has the form 
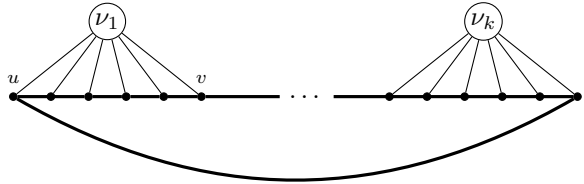
\[
    \begin{tikzpicture}
        \node[int, label=90:{$\scriptstyle u$}] (u) at (0,0) {}; 
        \node[int] (u1) at (.5,0) {}; 
        \node[int] (u2) at (1,0) {}; 
        \node[int] (u3) at (1.5,0) {}; 
        \node[int] (u4) at (2,0) {}; 
        \node[int, label=90:{$\scriptstyle v$}] (u5) at (2.5,0) {}; 
        \node[ext] (nu1) at (1.25,1) {$\nu_1$};
        \node (dd) at (3.9,0) {$\dots$};
        \begin{scope}[xshift=5cm]
        \node[int] (bu) at (0,0) {}; 
        \node[int] (bu1) at (.5,0) {}; 
        \node[int] (bu2) at (1,0) {}; 
        \node[int] (bu3) at (1.5,0) {}; 
        \node[int] (bu4) at (2,0) {}; 
        \node[ext] (nuk) at (1.25,1) {$\nu_k$};
        \node[int] (bu5) at (2.5,0) {};
        \draw (nuk) edge (bu1) edge (bu2) edge (bu3) edge (bu4) edge (bu5) edge (bu);
        \end{scope}
        \draw[very thick] (u) edge[bend right] (bu5) edge (u2) (u3) edge (u2) edge (u4) (u4) edge (u5) (dd) edge (u5) edge (bu5);
        \draw (nu1) edge (u1) edge (u2) edge (u3) edge (u4) edge (u5) edge (u);
    \end{tikzpicture}
    \]
    with $\nu_1,\dots, \nu_k$ trees. Consider one such tree, say $\nu_1$.
    The endpoints $u,v$ on our cycle are connected by a unique path $P$ in $\nu_1$. 
    By a similar argument as above, we may use the IHX relations to move all vertices of the tree onto the path $P$. Doing the same for all components $\gamma_1,\dots,\gamma_k$ we have hence reduced to Case 3).

    Case 5): Suppose that $\Gamma$ has the form of Case 4) above, but now the legs of the trees $\nu_j$ trees might attach in arbitrary order to the base loop $S$. Then using the IHX relations at the base loop $S$ again to disentangle the legs at the base loop, we reduce to Case 4). (Note that $k$ will typically be reduced in the process.)

    Case 6): Finally suppose that $\Gamma$ has the general form \eqref{equ:std form} with $\nu$ possibly disconnected and/or or higher loop order. Then we may use the IHX relations to move all vertices onto the base loop $S$, whence the graph is necessarily of the form of Case 6). (In this case with each $\nu_j$ having only one edge.)
\end{proof}
\begin{rem}
Note that for even $n$ we need to exclude the case $g=3$ for the following reason. For both even and odd $n$,  $H(\GC_n^{3\lp})$ is one-dimensional and generated by the tetrahedron graph cocycle
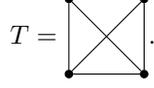
\[
T=
\begin{tikzpicture}[yshift=-.5cm]
    \node[int] (v1) at (0,0) {};
    \node[int] (v2) at (1,0) {};
    \node[int] (v3) at (0,1) {};
    \node[int] (v4) at (1,1) {};
    \draw (v1) edge (v2) edge (v3) edge (v4)  
    (v2) edge (v3) edge (v4) 
    (v3) edge (v4);
\end{tikzpicture}.
\]
For $n$ odd $T$ is cohomologous to a multiple of the graph 
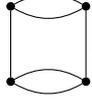
\[
\begin{tikzpicture}[yshift=-.5cm]
    \node[int] (v1) at (0,0) {};
    \node[int] (v2) at (1,0) {};
    \node[int] (v3) at (0,1) {};
    \node[int] (v4) at (1,1) {};
    \draw (v1) edge[bend left] (v2) edge[bend right] (v2) edge (v3)  
    (v2) edge (v4) 
    (v3) edge[bend left] (v4) edge[bend right] (v4);
\end{tikzpicture},
\]
which is the only barrel graph in loop order 3.
This latter graph is however zero for $n$ even by symmetry. The reduction to barrel diagrams is hence not possible for even $n$ and $g=3$.
\end{rem}

Next, let 
\[
B_g=\vspan \{B_\pi \mid \pi \in S_{g-1}\}\subset \GC_n^{g\lp}
\]
be the subspace spanned by barrel diagrams. 
Let $V\subset \GC_n^{g\lp, top-1}$ be an arbitrary subspace of elements of degree top-1.
Then the following is obvious:
\begin{cor}\label{cor:upper kneissler}
For any $n$ and $g\geq 4$ and any choice of subspace $V\subset \GC_n^{g\lp, top-1}$ we have that
\[
\vdim H^{top}(\GC_n^{g\lp}) \leq \vdim B_g - \mathrm{rank} (d\mid_{d^{-1}B_g}),
\]
where $d^{-1}B_g$ is the preimage of $B_g$ under the differential $d:V\to \GC_n^{g\lp, top}$.
\end{cor}

Now we choose for $V$ the span of the following graphs:
\begin{itemize}
    \item Barrel-like graphs with one vertex 4-valent:
    \[
X_\pi :=
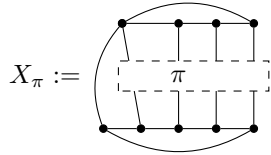
\begin{tikzpicture}
    \node[int] (v1) at (0.25,.7) {};
    \node[int] (v2) at (.25,.7) {};
    \node[int] (v3) at (1,.7) {};
    \node[int] (v4) at (1.5,.7) {};
    \node[int] (v5) at (2,.7) {};
    \node[int] (w1) at (0,-.7) {};
    \node[int] (w2) at (.5,-.7) {};
    \node[int] (w3) at (1,-.7) {};
    \node[int] (w4) at (1.5,-.7) {};
    \node[int] (w5) at (2,-.7) {};
    \draw (v1) edge[bend left] (v5) (w1) edge[bend right] (w5)
    (v2) edge (v1) edge (v2) (v3) edge (v2) edge (v4) (v4) edge (v5)
    (w2) edge (w1) edge (v2) edge (w2) (w3) edge (v3) edge (w2) edge (w4) (w4) edge (v4) edge (w5) (w1) edge[bend left] (v1) (w5) edge (v5);
    \draw[dashed, fill=white] (0.2,-.2) rectangle (2.2,.2);
    \node at (1,0) {$\pi$};
\end{tikzpicture}.
\]
\item Barrel-like graphs of the form
    \[
Y_\pi :=
\begin{tikzpicture}
    \node[int] (l) at (-.5,-.7) {};
    \node[int] (v1) at (0,.7) {};
    \node[int] (v2) at (.5,.7) {};
    \node[int] (v3) at (1,.7) {};
    \node[int] (v4) at (1.5,.7) {};
    \node[int] (v5) at (2,.7) {};
    \node[int] (w1) at (0,-.7) {};
    \node[int] (w2) at (.5,-.7) {};
    \node[int] (w3) at (1,-.7) {};
    \node[int] (w4) at (1.5,-.7) {};
    \node[int] (w5) at (2,-.7) {};
    \draw (v1) edge[bend left] (v5) (w1) edge[bend right] (w5) edge (l)
    (v2) edge (v1) edge (v2) (v3) edge (v2) edge (v4) (v4) edge (v5)
    (w2) edge (w1) edge (v2) edge (w2) (w3) edge (v3) edge (w2) edge (w4) (w4) edge (v4) edge (w5) (l) edge[bend left] (v1) edge[bend right] (v1) (w5) edge (v5);
    \draw[dashed, fill=white] (-.2,-.2) rectangle (2.2,.2);
    \node at (1,0) {$\pi$};
\end{tikzpicture}.
\]
\end{itemize}
Hence define
\begin{equation}\label{equ:V def}
V_g:=\vspan \{ X_\pi, Y_\pi \mid \pi \in S_{g-2}\}.
\end{equation}

\subsection{Computational aspects and results}
In practice, we will compute $\mathrm{rank} (d\mid_{d^{-1}B_g})$ as follows. Pick some complement $B_g^\perp$ of $B_g$ in $\GC_n^{g\lp, top}$ and decompose $d:V\to \GC_n^{g\lp, top}$ accordingly as 
\[
d = \begin{pmatrix}
    d^\parallel \\
    d^\perp
\end{pmatrix}.
\]
Then we have that 
\[
\mathrm{rank} (d\mid_{d^{-1}B_g}) = \mathrm{rank}(d) -\rk(d^\perp).
\]
Furthermore it is actually not necessary to take a full complement $B_g^\perp$, but just the span of the graphs that can actually be produced by the differential from the generators of $V_g$. For our choice of $V_g$ as in \eqref{equ:V def} we can hence set 
\begin{equation}\label{equ:compl}
B_g^\perp = \vspan\left(\{ A_\pi \mid \pi\in S_{g-2}, \nexists \pi': A_\pi \cong B_{\pi'} \} \cup \{A_\pi' \mid \pi\in S_{g-2}, \nexists \pi': A_\pi' \cong B_{\pi'} \} \right)
\end{equation}
with 
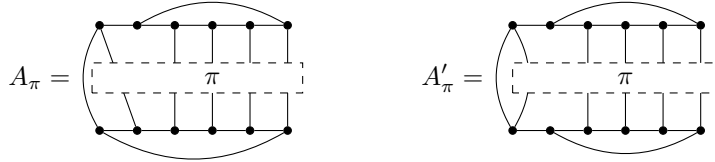
\begin{align*}
A_\pi &= 
    \begin{tikzpicture}
    \node[int] (w0) at (-.5,-.7) {};
    \node[int] (v0) at (-.5,.7) {};
    \node[int] (v1) at (0,.7) {};
    \node[int] (v2) at (.5,.7) {};
    \node[int] (v3) at (1,.7) {};
    \node[int] (v4) at (1.5,.7) {};
    \node[int] (v5) at (2,.7) {};
    \node[int] (w1) at (0,-.7) {};
    \node[int] (w2) at (.5,-.7) {};
    \node[int] (w3) at (1,-.7) {};
    \node[int] (w4) at (1.5,-.7) {};
    \node[int] (w5) at (2,-.7) {};
    \draw (v1) edge[bend left] (v5) edge (v0) 
    (w1) edge (v0)
    (w0) edge[bend right] (w5) edge (w1)
    (v2) edge (v1) edge (v2) (v3) edge (v2) edge (v4) (v4) edge (v5)
    (w2) edge (w1) edge (v2) edge (w2) (w3) edge (v3) edge (w2) edge (w4) (w4) edge (v4) edge (w5) 
    (w0) edge[bend left] (v0) (w5) edge (v5);
    \draw[dashed, fill=white] (-.6,-.2) rectangle (2.2,.2);
    \node at (1,0) {$\pi$};
\end{tikzpicture}
&
A_\pi' &= 
    \begin{tikzpicture}
    \node[int] (w0) at (-.5,-.7) {};
    \node[int] (v0) at (-.5,.7) {};
    \node[int] (v1) at (0,.7) {};
    \node[int] (v2) at (.5,.7) {};
    \node[int] (v3) at (1,.7) {};
    \node[int] (v4) at (1.5,.7) {};
    \node[int] (v5) at (2,.7) {};
    \node[int] (w1) at (0,-.7) {};
    \node[int] (w2) at (.5,-.7) {};
    \node[int] (w3) at (1,-.7) {};
    \node[int] (w4) at (1.5,-.7) {};
    \node[int] (w5) at (2,-.7) {};
    \draw (v1) edge[bend left] (v5) edge (v0) 
    (w1) edge[bend right] (w5) 
    (w0) edge (w1)
    (v2) edge (v1) edge (v2) (v3) edge (v2) edge (v4) (v4) edge (v5)
    (w2) edge (w1) edge (v2) edge (w2) (w3) edge (v3) edge (w2) edge (w4) (w4) edge (v4) edge (w5) 
    (w0) edge[bend left] (v0) edge[bend right] (v0)
    (w5) edge (v5);
    \draw[dashed, fill=white] (-.5,-.2) rectangle (2.2,.2);
    \node at (1,0) {$\pi$};
\end{tikzpicture}.
\end{align*}
Note that in the definition of our complement \eqref{equ:compl} we have to make sure to exclude graphs that are isomorphic to barrel graphs by accidental symmetry.

Next, we do not need to compute $\rk(d^\perp)$ explicitly by the following result:
\begin{lemma}
    The map $d^\perp : V_g\to B_g^\perp$ is surjective, hence $\rk(d^\perp)=\vdim B_g^\perp$.
\end{lemma}
\begin{proof}
The differential of $X_\pi$ is $\pm A_\pi$ plus barrel graphs. The differential of $Y_\pi$ is $A_\pi'$, plus graphs of the form $A_\pi$. Hence the lemma follows.    
\end{proof}
We hence obtain for our upper bound on the top cohomology
\[
\vdim H^{top}(\GC_n^{g\lp}) \leq \vdim B_g - \rk(d) +\vdim B_g^\perp. 
\]
The next issue to be addressed is that for memory reasons we can only compute matrix ranks over finite fields $\F_p$, at least for the larger values of $g$. However, since passing to a finite field can only decrease an integer matrix's rank, and since the rank only enters with negative sign in the above formula, we finally find that
\[
\vdim_{\Q} H^{top}(\GC_n^{g\lp}) \leq \vdim B_g - \rk_{\F_p}(d) +\vdim B_g^\perp. 
\]

We computed the dimensions of the spaces $B_g$ and $B_g^\perp$ as well as the $\F_p$-rank of $d:V\to B_g\oplus B_g^\perp$. The results are displayed in Figure \ref{fig:kneissler results}.
\begin{figure}
\begin{tabular}{c|c|c|c|c|c|}
$g$ & $\vdim B_g$ & $\vdim B_g^\perp$ & $\vdim V_g$ & $\rk d$ & upper bound $H^{top}(\GC_n^{g\lp})$\\
\hline
& \multicolumn{1}{c}{\bf even $n$} \\
\hline
5  &  0  &  0  &1&  0   &  0  \\
6  &  2  &  2  &4&  3   &  1  \\
7  &  2  &  6  &19&  8   &  0  \\
8  &  6  &  39  &143&  45   &  0  \\
9  &  66  &  369  &1237&  435   &  0  \\
10  &  542  &  3484  &10652&  4025  &  1  \\
11 & 4160 & 34485 &101450& 38645&  0\\
12 & 38255 & 381132 &1067721& 419387&  0\\
13 & 396187 & 4566456 &12213628& 4962642  & 1 \\
14 & 4433632 & 58688199 &150806455& 63121830 & 1\\
\hline
& \multicolumn{1}{c}{\bf odd $n$} \\
\hline
5  &  2  &  1  &1&  1  &    2  \\
6  &  3  &  3  &4&  4  &    2  \\
7  &  9  &  13  &27&  19    &  3  \\
8  &  27  &  65  &167&  88    &  4  \\
9  &  121  &  443  &1303&  559    &  5  \\
10  &  652  &  3778  &11091&  4424    &  6  \\
11 & 4779 & 36217 &104692& 40988& 8\\
12 & 41229 & 391291 &1092083& 432511  & 9 \\
13 & 410706 & 4636086 &12431753& 5046781  & 11 \\
14 & 4533035 & 59385487 &153050537& 63918509  & 13\\
\hline
\end{tabular}
\caption{\label{fig:kneissler results} Numerical results for the upper bound on $H^{top}(\GC_n^{g\lp})$ according to Corollary \ref{cor:upper kneissler}. For $g\geq 11$ the rank of $d$ is computed over a finite field ($\F_{3323}$). We nevertheless get valid bounds on the dimension of the cohomology with rational coefficients as discussed in the text. }
\end{figure}
Note that the resulting matrices of $d$ are very large with size tens of millions in the most difficult cases. However, they are also very sparse.
To compute their rank we use sparse Gaussian elimination with the following custom pivoting heuristic: We first try to choose the pivots so as to kill the matrix rows corresponding to elements of $B_g^\perp$. Then we try to choose pivots to eliminate the matrix columns corresponding to elements $X_\pi$.
Overall, this pivoting heuristic produces an algorithm that is roughly analogous to the ``ad hoc'' algorithm of \cite{KneisslerNumber}. Note however that we fully take into account graph isomorphisms, whereas loc. cit. only treats the $B_\pi$ as formal symbols, modulo some (ad hoc) isomorphisms.

In the cases of odd $n$ our upper bound agrees with a lower bound found by Vogel on the top cohomology, cf. \cite[Table 1]{BroadhurstV}, \cite{Vogel}, and hence we have computed the exact top cohomology dimension, as appears in Table \ref{tab:GC3}.
For even $n$ there is no known lower bound on the top cohomology dimension unfortunately, hence our upper bounds are recorded as such in Table \ref{tab:main res even n}. (Except of course in the case $g=12$ when the upper bound is zero.)

Finally, we observe that in all computed cases where we know the dimension of the top cohomology our upper bound is actually sharp. Hence let us raise the following conjecture: 
\begin{conj}
    For the choice of subspace $V=V_g$ as in \eqref{equ:V def} the dimension bound of Corollary \ref{cor:upper kneissler} is sharp.  
\end{conj}

\section{Details on the numerical computations in loop order 11}
To compute the dimension of $H^k(\G_n^{g\lp})$ we use the slightly smaller triconnected graph complex $\G_n^{tri}$, which has been shown in \cite{WillwacherTriconnected} to be quasi-isomorphic to $\G_n$. The complex $\G_n^{tri}$ is defined analogous to $\G_n$, but requiring all graphs to be simple and triconnected.

Let $d_{g,k}$ denote the differentials in that complex.
\[
\cdots \to \G_n^{tri,g\lp,-k} \xrightarrow{d_{g,k}} \G_n^{tri,g\lp,-k+1} \to \cdots
\]
Then we have 
\begin{equation}\label{equ:dim H}
\vdim H^k(\GC_n^{g\lp}) = \vdim H^{-k}(\G_n^{tri,g\lp})
=
\vdim \G_n^{tri,g\lp,-k} - \rk d_{g,k} - \rk d_{g,k+1}.
\end{equation}
We hence compute a basis of the vector spaces $\G_n^{tri,g\lp}$, the sparse matrices $d_{g,k}$ and their ranks. The rank computation is by far the most time consuming step, owed to the large dimension of the vector spaces, see Table \ref{tab:comp res}.
\begin{table}
\begin{tabular}{c|c|c|c}
    \multicolumn{1}{c}{$n=2$}
    \\\hline
    $k$ & $\vdim\G_2^{tri,11\lp,-k}$ & $\rk d_{11,k}$ & $\vdim H^k(\GC_2^{11\lp})$ \\
    \hline
    0 & 4305302	& (known) & 2\\
    1 & 12906363 & (known) & 0\\
    2 & 26371100 & (known) & 0\\
    3 & 37489748 & (known) & 2\\
    4 & 37153253 & (known) & 0 \\	
    5 & 25209403 &	16611125 & 0\\
    6 & 11176863 &	8598278 & 1\\
    7 & 2919555	& 2578584 & 0 \\
    8 & 340971 & 340971 & 0	\\
    \hline 
    \multicolumn{1}{c}{$n=3$}
    \\\hline
    $k$ & $\vdim\G_3^{tri,11\lp,-k}$ & $\rk d_{11,k}$ & $\vdim H^k(\GC_3^{11\lp})$ \\
    \hline
    -11& 4329527 & (known) & 0 \\
    -10& 12941768 & 3500993& 0 \\
    -9& 26392022 & 9440775& 1 \\
    -8&	37472645 & 16951246& 0 \\
    -7&	37062161 & 20521399& 1 \\
    -6& 25057042 & 16540761& 7 \\
    -5&	11066860 & 8516274& 0 \\
    -4&	2894436 & 2550586& 0 \\
    -3 & 343858 & 343850& 8
\end{tabular}
    \caption{\label{tab:comp res} Dimensions of the vector spaces $\G_n^{tri,11\lp,-k}$ and the computed (lower bounds on the) ranks of the differentials $d_{g,k}$ and (upper bounds on the) cohomology dimensions. The ranks marked "(known)" have not been computed numerically, because they can be inferred theoretically, see the arguments in section \ref{sec:cohom discussion} above.}
\end{table}
For the rank computation we work over a finite field $\F_p$ and use a (symmetric) block Coppersmith-Wiedemann algorithm. This algorithm computes a (likely correct) lower bound the rank of a matrix $A$ in two steps: 
\begin{itemize}
    \item Step 1) Choose a number $N$ and $N$ random vectors $u_1,\dots,u_N$. Compute the sequence of $N\times N$ matrices $(u_{\alpha}^tB^ku_{\beta})_{\alpha,\beta=1,\dots,N}$ for $B$ a preconditioned version of $A^TA$.
    \item Step 2): Compute a generating polynomial for this matrix sequence and use it to determine a lower rank bound.
\end{itemize}
The algorithm is well-known, and versions of it have been used for similar problems in the past \cite{DumasEtal}. An implementation can also be found in the Linbox library \cite{Linbox}. 
In practice, the computational bottleneck for large matrices is the sequence computation in Step 1). Our main new contribution is that we use an optimized GPU (CUDA) implementation for this sequence computation (matrix-vector products modulo $p$), which reduced the runtimes from several months to several days compared to a CPU implementation. In practice, we used used values of $N$ between $4$ and $32$.
Step 2) of the algorithm is unchanged compared to prior art, and still runs on the CPU. However, it takes far less time than Step 1) -- a few hours for the largest matrices.

It is important to note that the Coppersmith-Wiedemann algorithm only produces a(n albeit likely correct) lower bound on the $\F_p$-rank of a matrix, and furthermore, the $\F_p$-rank is only a lower bound to the $\Q$-rank. Hence via \eqref{equ:dim H} our computation results a priori only in an upper bound on the ($\Q$-)dimension of the cohomology.
However, if the upper bound happens to be zero, then we know that $H^k(\GC_n^{g\lp})=0$ and moreover we learn that our lower bounds on  $\rk d_{g,k}$ and $\rk d_{g,k+1}$ were both tight. Hence all entries in Tables \ref{tab:GC2} and \ref{tab:GC3} that are zero or have two adjacent zero entries are actually correct and not just an upper bound. Fortunately, this holds for all entries computed in this manner, except possibly the entries $(11,-6)$, $(11,-7)$ in Table \ref{tab:GC3}. These entries are hence to be considered (likely correct) upper bounds, or more precisely the dimensions are either $7$ and $1$ as computed or $6$ and $0$.

\section{Representatives of cohomology classes}
\label{sec:reps}
We now focus solely on the cohomology of $\GC_2$, see Table \ref{tab:GC2}. We desire to write down fairly explicit representatives for some of the cohomology classes.
To this end it will be easier to work with an extended version $\GC_2^2$ of the Kontsevich graph complex, which is defined by merely allowing bivalent vertices and self-edges in graphs.
We recall the following well-known results: Obviously $\GC_2\subset \GC_2^2$ is a dg Lie subalgebra. 
We furthermore have a decomposition into a direct sum of dg vector spaces 
\begin{equation}\label{equ:GC2 direct}
\GC_2^2 \cong \GC_2 \oplus X \oplus \bigoplus_{k\geq 0} \Q L_{4k+1},
\end{equation}
where $L_{4k+1}$ is the loop graph with $4k+1$ edges and vertices 
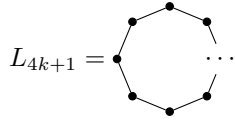
\[
L_{4k+1} = \begin{tikzpicture}
    \node (v0) at (0:.7) {$\cdots$};
    \node[int] (v1) at (45:.7) {};
    \node[int] (v2) at (90:.7) {};
    \node[int] (v3) at (135:.7) {};
    \node[int] (v4) at (180:.7) {};
    \node[int] (v5) at (225:.7) {};
    \node[int] (v6) at (270:.7) {};
    \node[int] (v7) at (315:.7) {};
    \draw (v1) edge (v2) edge (v0) 
    (v3) edge (v2) edge (v4)
    (v5) edge (v6) edge (v4) 
    (v7) edge (v0) edge (v6);
\end{tikzpicture}
\]
and $X$ is the subcomplex spanned by graphs of loop order $\geq 2$ that contain at least one bivalent vertex. This subcomplex $X$ is acyclic, $H(X)=0$, so that
\[
H(\GC_2^2) \cong H(\GC_2) \oplus \bigoplus_{k\geq 0} \Q L_{4k+1}.
\]
Furthermore, due to the direct sum decomposition \eqref{equ:GC2 direct} we have a projection of dg vector spaces 
\[
\pi : \GC_2^2 \to \GC_2.
\]
This map is not compatible with the Lie brackets, though.
Furthermore, there is a (complete) grading on $\GC_2^2$ and $\GC_2$ by the loop order, and all morphisms considered respect that grading.

Next we consider the variant $\GC_0^2$ of the graph complex. It satisfies
\[
\GC_0^{2, g\lp} = \GC_2^{2, g\lp}[-2g].
\]
In this dg Lie algebra the element $L_1=\tadpole$ is a Maurer-Cartan element, so that can consider the twisted dg Lie algebra 
\[
(GC_0^2, \delta + [L_1,-]).
\]
We shall write $\nabla := [L_1,-]$ for short. There is still a descending complete filtration by loop order on the twisted graph complex. The spectral sequence with respect to this filtration converges to a one-dimensional vector space
\[
E^1 = H(\GC_0^2, \delta) \Rightarrow \Q L_1,
\]
as shown in \cite{KWZ}.
This is the spectral sequence that we have already used in section \ref{sec:cohom discussion} above. 
It is furthermore shown in \cite{KWZ} that in this spectral sequence the element $L_{4k+1}$ (for $k\geq 1$) cancels the generators of the Grothendieck-Teichmüller group $\sigma_{2k+1}$ in 
\[
H^0(\GC_2^{2k+1\lp})\cong H^{4k+2}(\GC_0^{2k+1\lp})
\]
on some page of the spectral sequence.
Concretely, this implies that there are elements 
\[
\hat L_{4k+1} = L_{4k+1} + (\cdots),
\]
with the terms $(\cdots)$ of loop order $\geq 2$, so that 
\[
(\delta + \nabla) \hat L_{4k+1} = (const) W_{2k+1} + (\cdots) =: \sigma_{2k+1},
\]
with $W_k$ the $2k+1$-wheel graph, $(const)$ a nonzero constant and $(\cdots)$ a linear combination of other graphs of loop order $2k+1$.
One can use the elements $\hat L_{4k+1}$ and $\sigma_{2k+1}$ to give explicit formulas for most of the cohomology classes appearing in Table \ref{tab:GC2}.
First, the entries of the 0-th row $H^0(\GC_2)$ can be identified with the free Lie algebras in the symbols $\sigma_{2k+1}$, $k=1,2,\dots$.
To get the representative for the first class of non-zero degree, located at position $(g,k)=(6,3)$, note that this class has to cancel $[\sigma_3,\sigma_5]$ in our spectral sequence.
However, the element $[\hat L_5, \sigma_5]$ is concentrated in loop orders $\geq 6$ and satisfies 
\[
(\delta + \nabla) [\hat L_5, \sigma_5] = [(\delta + \nabla)\hat L_5, \sigma_5]
= [\sigma_3,\sigma_5].
\]
It follows that the trivalent loop order 6 part 
\[
X_6 := \pi ([\hat L_5, \sigma_5])_6 \in \GC_0^{6\lp} \cong \GC_2^{6\lp}[-4]
\]
is a $\delta$-cocycle, $\delta X_6=0$, and hence defines a cohomology class in $H^3(\GC_2^{6\lp})$.
If this element was $\delta$-exact, then by adding a coboundary to $[\hat L_5, \sigma_5]$ we could make an element $Y$ of loop order $\geq 7$ such that $(\delta + \nabla)Y=[\sigma_3,\sigma_5]$. But since $H^1(\GC_2^{7\lp})=0$ we could similarly even take $Y$ of loop order $\geq 8$. But that would make $[\sigma_3,\sigma_5]$ a cocycle, in contradiction to its non-triviality. We hence conclude that $X_6$ is a non-trivial cohomology class.

In a similar manner representatives for the other classes of degree 3 in Table \ref{tab:GC2} can be found. 
We want to go a bit further and consider generally the elements $[\hat L_{4k+1}, \sigma_{4k+1}]$ and their trivalent leading terms 
\[
X_{4k+2} = \pi ([\hat L_{4k+1}, \sigma_{4k+1}])_{4k+2}.
\]
Concretely, we have that 
\[
X_{4k+2} = (const) 
\sum_{\sigma \in S_{4k+1}}
\begin{tikzpicture}
    \node[int] (v1) at (0,.7) {};
    \node[int] (v2) at (.5,.7) {};
    \node[int] (v3) at (1,.7) {};
    \node[int] (v4) at (1.5,.7) {};
    \node[int] (v5) at (2,.7) {};
    \node[int] (w1) at (0,-.7) {};
    \node[int] (w2) at (.5,-.7) {};
    \node[int] (w3) at (1,-.7) {};
    \node[int] (w4) at (1.5,-.7) {};
    \node[int] (w5) at (2,-.7) {};
    \draw (v1) edge[bend left] (v5) (w1) edge[bend right] (w5)
    (v2) edge (v1) edge (v2) (v3) edge (v2) edge (v4) (v4) edge (v5)
    (w2) edge (w1) edge (v2) edge (w2) (w3) edge (v3) edge (w2) edge (w4) (w4) edge (v4) edge (w5) (w1) edge (v1) (w5) edge (v5);
    \draw[dashed, fill=white] (-.2,-.2) rectangle (2.2,.2);
    \node at (1,0) {$\sigma$};
\end{tikzpicture}
=
(const) \sum_{\sigma\in S_{4k+1}}B_\sigma.
\]
where $(const)$ is an irrelevant nonzero constant and we use the notation $B_\sigma$ for the barrel graphs as in \eqref{equ:Bpi def} above.

\begin{conj}
For every $k=1,2,\dots$ the element $X_{4k+2}$ defines a non-trivial cohomology class in $H^{4k-1}(\GC_2^{4k+2\lp})$. 
\end{conj}

The conjecture holds for $k=1$ as we checked above.
By computer calculation we also checked that it holds for $k=2$.
Hence we deduce that the element $X_{10}$ generating $H^7(\GC_2^{10\lp})$ can be extended by adding terms of higher loop order to an element $\hat X_{10}$ such that $(\delta+\nabla)\hat X_{10}$ is of loop order $\geq 14$.
Rewording this statement, we have shown:
\begin{cor}\label{cor:X10 survives}
In our spectral sequence the generator of $H^7(\GC_2^{10\lp})\cong \Q X_{10}$ survives until the $E^{4}$-page, where it kills the element $[\sigma_{5}, \sigma_{9}]\in H^0(\GC_2^{14\lp})$.  
\end{cor}

\section{On a conjecture of Francis Brown}
It is known by work of Francis Brown \cite{Brown} the Grothendieck-Teichmüller Lie algebra $\grt_1\cong H^0(\GC_2)$ contains a free Lie algebra in generators $\sigma_{2k+1}$.\footnote{To be precise, Brown's generators are not known to agree with our $\sigma_{2k+1}$ in general, though this is true for $g\leq 29$.} Brown furthermore conjectured \cite[Conjecture 2.5]{BrownGC} that this free Lie algebra can be extended to a larger free Lie algebra involving generators of higher degrees as well. The construction of the generators is a bit too involved to recall here, and partially conjectural. However, we note that $\sigma_3$ and $X_{10}$ are part of the conjectured generators. Thus we can provide a counterexample to the conjecture as follows:
\begin{prop}
    $[\sigma_3,X_{10}]=0$ as an element of $H^7(\GC_2^{13\lp})$.
\end{prop}
\begin{proof}
    $X_{10}$ is cohomologous to $X_{10}':=([\hat L_{9}, \sigma_9])_{10}$, so we may equivalently show that $Z:=[\sigma_3,X_{10}']=0$.
    Next, extend the $\delta$-cocycle $Z$ to an element 
    \[
    \hat Z = [\sigma_3,[\hat L_9, \sigma_9 ] ] 
    - \underbrace{ [\hat L_5, [\sigma_5, \sigma_9]] }_{\geq 15\text{-loops}},
    \]
    only adding terms of higher loop order in the process.
    The element $\hat Z$ then satisfies 
    \[
    (\delta+\nabla ) \hat Z = [\sigma_3,[(\delta+\nabla )\hat L_9, \sigma_9 ] ]
    -[(\delta+\nabla ) \hat L_5, [\sigma_5, \sigma_9]]
    =
    [\sigma_3,[\sigma_5, \sigma_9 ] ]-[\sigma_3,[\sigma_5, \sigma_9 ] ]=0.
    \]
    Hence $\hat Z$ must be a $(\delta+\nabla)$-cocycle, i.e., there is an element $Y\in\GC_0^2$ such that 
    \[
    \hat Z=(\delta+\nabla ) Y.
    \]
    However, if $Z$ were a non-trivial cohomology class, this would imply that it is killed by some other class in our spectral sequence.
    However, this is not possible by degree reasons, see Table \ref{tab:GC2}: The first possible position $(11,9)$ from where it could be killed is already in the vanishing range.
\end{proof}
Despite the negative result above, a version of Brown's conjecture might still be true, in that there is still a nonzero class in $H(\GC_2)$ which the leading-loop-order term of an element $(\delta+\nabla)$-cohomologous to $\hat Z$. In fact, the proof above suggest that this class should live in $H^3(\GC_2^{15\lp})$.  
However, this ``descent'' of classes makes the structure of the Lie algebra $H(\GC_2)$ more difficult to understand.

\end{document}